\documentclass[11pt]{amsart}
\usepackage{mathrsfs} 
\usepackage[margin=1.25in]{geometry}
\usepackage{ifthen}
\usepackage{graphicx}
\usepackage{epsfig}
\usepackage{dsfont, amssymb}
\usepackage{amsfonts,dsfont,amssymb,amsthm,stmaryrd,bbm}
\usepackage{amsmath}
\usepackage{verbatim}
\newcommand{\dd}{\mathrm{d}}
\DeclareMathOperator{\var}{Var}

\usepackage{xcolor}

\usepackage[toc,page]{appendix} 
\usepackage{hyperref,mathtools}
\hypersetup{colorlinks=true,pdftitle="",pdftex}

\let\originallesssim\lesssim
\let\originalgtrsim\gtrsim

\DeclareRobustCommand{\lesssim}{%
  \mathrel{\mathpalette\lowersim\originallesssim}%
}
\DeclareRobustCommand{\gtrsim}{%
  \mathrel{\mathpalette\lowersim\originalgtrsim}%
}

\makeatletter
\newcommand{\lowersim}[2]{%
  \sbox\z@{$#1<$}%
  \raisebox{-\dimexpr\height-\ht\z@}{$\m@th#1#2$}%
}
\makeatother

\newtheorem{thm}{Theorem}[section]

\newtheorem{lem}[thm]{Lemma}
\newtheorem{prop}[thm]{Proposition}

\newtheorem{defn}[thm]{Definition}
\newtheorem{cor}[thm]{Corollary}

\newcommand\independent{\protect\mathpalette{\protect\independent}{\perp}} 
\def\independent#1#2{\mathrel{\rlap{$#1#2$}\mkern2mu{#1#2}}}

\newcommand{\mc}[1]{\mathcal{{#1}}}
\newcommand{\supp}{\mathrm{Supp}}

\newcommand{\R}{\mathbb{R}}

\newcommand{\mE}{\mathbb{E}}

\newcommand{\e}{\varepsilon}

\def\Var{{\rm Var}}

\def\phi{\varphi}

\def\bee{\begin{eqnarray*}}
\def\ene{\end{eqnarray*}}
\usepackage{mathtools}

\begin{document}

\title{Working File}

\title{Minimum entropy of a log-concave variable  with fixed variance}

\author{James Melbourne, Piotr Nayar,  and Cyril Roberto}

\address{Department of Probability and Statistics, Centro de Investigación en
matemáticas (CIMAT)}
\email{james.melbourne@cimat.mx}

\address{University of Warsaw, 02-097 Warsaw, Poland}
\email{nayar@mimuw.edu.pl}

\address{MODAL'X, UPL, Univ. Paris Nanterre, CNRS, F92000 Nanterre France}
\email{croberto@math.cnrs.fr}

\thanks{P.N.’s research was supported by the National Science
Centre, Poland, grant 2018/31/D/ST1/01355. C.R.'s research was supported by the Labex MME-DII funded by ANR, reference ANR-11-LBX-0023-01
and the FP2M federation (CNRS FR 2036).}

\begin{abstract}
We show that for log-concave real random variables with fixed variance the Shannon differential entropy is minimized for an exponential random variable. We apply this result to derive upper bounds on capacities of additive noise channels with log-concave noise. We also improve constants in the  reverse entropy power inequalities for log-concave random variables.  
\end{abstract}



\maketitle

\section{Introduction}

For a real random variable  $X$ with density $f$ its  differential entropy is defined via the formula $h(X)=h(f)=-\int f \log f$. This definition goes back to the celebrated work of Shannon \cite{Sha48}, but the same quantity was also considered, without the minus sign, by physicists, including Boltzmann, in the context of thermodynamics of gases. In fact, it is a classical fact going back to Boltzmann \cite{Bolt:lectures} that under fixed variance the entropy is maximized for a Gaussian random variable. This leads to the translation and scale invariant inequality 
\[
    h(X) \leq \frac12 \log \var(X) + \frac12 \log(2\pi e). 
\]
One can see that in general one cannot hope for a reverse bound, since for the density $f_\e(x)=(2\e)^{-1} \mathbf{1}_{[1,1+\e]}(|x|)$ the variance stays bounded while the entropy goes to $-\infty$ as $\e \to 0^+$. However, a reverse bound still holds if one imposes some extra assumption on $X$, such as log-concavity. Recall that $X$ is said to be log-concave if its density  is of the form $f=e^{-V}$, where $V:\R \to (-\infty,\infty]$ is convex.  In \cite{BM11:it, marsiglietti2018lower} the authors showed that indeed in this class the inequality can be reversed. In particular, in \cite{marsiglietti2018lower} it was proved that if $X$ is log-concave, then $h(X) \geq \frac12 \log \var(X)+\log 2$. The constant $\log 2$ is suboptimal and it became a well-known open problem to find the optimal bound. The goal of this article is to prove the following optimal inequality.   

\begin{thm}\label{thm:main}
    For a log-concave random variable $X$ we have
    \[
        h(X) \geq  \frac12 \log \var(X) +1
    \]
with equality for the standard one-sided exponential random variable with density $e^{-x}\mathbf{1}_{[0,\infty)}(x)$.
\end{thm}

\noindent Probably the most significant generalization of entropy is the so-called  R\'enyi entropy of order $\alpha \in (0,\infty) \setminus \{1\}$ \cite{Ren61}, which is defined as 
\[
	h_\alpha(X)=h_\alpha(f) = \frac{1}{1-\alpha} \log\left( \int f^\alpha(x) \dd x \right),
\]  
assuming that the integral converges, see \cite{madiman2017forward} for more background. If $\alpha \to 1$ one recovers the usual Shannon differential entropy $h(f)=h_1(f)=-\int f \ln f$. Also, by taking limits one can put $h_0(f)=\log|\supp f|$, where $\supp f$ stand for the support of $f$ and $h_\infty(f)=-\log\|f\|_\infty$, where $\|f\|_\infty$ is the essential supremum of $f$. For $p \geq q>0$ one has
\[
    0 \leq h_q(f) - h_p(f) \leq \frac{\log q}{q-1} - \frac{\log p}{p-1},
\]
with equality in the right hand inequality for the standard one-sided exponential random variable. Here the fraction $\frac{\log t}{t-1}$ is interpreted as $1$ for $t=1$, see \cite{FMW16}. Thus using this bound together with Theorem \ref{thm:main} gives the following corollary.

\begin{cor}
For $\alpha>1$ and a log-concave random variable $X$ one has     
\[
 h_\alpha(X) \geq  \frac12 \log \var(X) + \frac{\log \alpha}{\alpha-1},
\]
with equality when $X$ is an exponential random variable.
\end{cor}
\noindent Let us mention that the problem of minimizing R\'enyi entropy under fixed variance for \emph{symmetric} log-concave random variables was solved in \cite{MNT21} for the case $\alpha \leq 1$ and in \cite{bialobrzeski2021r} for $\alpha>1$. Here the worst case is the uniform random variable on an interval for $\alpha \leq \alpha^*$ and symmetric exponential random variable for $\alpha > \alpha^*$, where $\alpha^* \approx 1.241$ is the solution to the equation $\frac{\log \alpha^* }{\alpha^*-1} = \frac12 \log 6$.


\section{Applications}

\subsection{Additive noise channels} We now briefly discuss an application of our main result in the context of information theory. For more details we refer the reader to \cite{madiman2021sharp}, where the case of symmetric log-concave random variables was discussed. Consider the memoryless transmission channel with power budged $P$ subject to additive noise $N$, that is, if the input of the channel is $X$, then output produced by the channel is $Y = X +N$,
where $N$ is the noise independent of $X$. Shannon's celebrated channel coding theorem \cite{Sha48} asserts that the so-called \emph{capacity} of such a channel is given by the formula
\[
	C_P(N) = \sup_{X: \ \var(X) \leq P} (h(X+N) - h(N)).
\]  
We have the following fact.
\begin{prop}
    Let $N$ be a random variable with finite variance and let $Z$ be a centered Gaussian random variable with the same variance. Then
\[
    C_P(Z) \leq C_P(N) \leq C_P(Z)+ D(N),
\]
where $D(N)=h(Z)-h(N)$ is the relative entropy of $N$ from Gaussianity.
\end{prop}
\noindent Our Theorem \ref{thm:main} gives
\[
   D(N) =  h(Z)-h(N) \leq h(Z) - \frac12 \log \var(N) -1 = \frac12 \log(2\pi e) -1 = \frac12 \log \left( \frac{2\pi}{e} \right).
\]
We can therefore formulate the following corollary.
\begin{cor}
     Let $N$ be a log-concave noise and let $Z$ be a centered Gaussian noise with the same variance. Then
\[
    C_P(Z) \leq C_P(N) \leq C_P(Z)+ \frac12 \log \left( \frac{2\pi}{e} \right).
\]
\end{cor}
\noindent It other words, using an arbitrary log-concave noise instead of the Gaussian does not increase capacity by more than   $\frac12 \log \left( \frac{2\pi}{e} \right)<0.42$ nats. 

\subsection{Reverse EPI}
Recall that the entropy power of a real random variable $X$ is defined by $\mc{N}(X)=\frac{1}{2\pi e} \exp(2h(X))$. Note that for a Gaussian random variable one has $\mc{N}(Z)=\var(Z)$ and in general $\var(X)=\var(Z)$ implies $\mc{N}(X) \leq \mc{N}(Z)$.
 Theorem \ref{thm:main} can be rewritten in the form $\mc{N}(X) \geq \frac{e}{2\pi} \var(X)$.  The entropy power inequality of Shannon and Stam \cite{Sha48, Sta59} states that for independent random variables $X,Y$ one has 
\[
\mc{N}(X+Y) \geq \mc{N}(X)+\mc{N}(Y).
\]
It is of interest to obtain reverse bounds. Under log-concavity assumption the authors of \cite{marsiglietti2018lower} showed that if $X,Y$ are uncorellated, then $\mc{N}(X+Y) \leq \frac{\pi e}{2} (\mc{N}(X)+\mc{N}(Y))$. Using Theorem \ref{thm:main} we can improve this result.
\begin{cor}
    Let $X,Y$ be log-concave uncorrelated real random variables. Then
    \[
    \mc{N}(X+Y) \leq  \frac{2\pi}{e} \left(\mc{N}(X)+\mc{N}(Y) \right). 
    \]
\end{cor}
\noindent Indeed, one has
\[
    \mc{N}(X+Y) \leq \var(X+Y) = \var(X) + \var(Y) \leq \frac{2\pi}{e} \left( \mc{N}(X)+\mc{N}(Y) \right).
\]
We can also define the R\'enyi entropy power via $\mc{N}_\alpha(X) = \exp(2h_\alpha(X))$, where we removed the normalizing constant $2\pi e$ for simplicity. Similarly as in \cite{bialobrzeski2021r}, Theorem 2 from \cite{LYZ05} together with our Theorem \ref{thm:main} gives the inequality  
\begin{equation} \label{ineq:ent-power}
	C_-(\alpha) \var(X) \leq N_\alpha(X) \leq C_+(\alpha) \var(X), \qquad \alpha>1
\end{equation}
where
\[
	C_-(\alpha) = \alpha^{\frac{2}{\alpha-1}}, \qquad \qquad
 C_+(\alpha) = \frac{3\alpha-1}{\alpha-1} \left( \frac{2\alpha}{3\alpha-1} \right)^{\frac{2}{1-\alpha}}   B\left(\frac12, \frac{\alpha}{\alpha-1}\right)^2 .
\] 
Here $B(x,y)=\frac{\Gamma(x)\Gamma(y)}{\Gamma(x+y)}$ stands for the Beta function. The  right inequality does not need log-concavity. Thus, using the same computation as for the case $\alpha=1$ we get the following corollary.
\begin{cor}
    If $X,Y$ are log-concave uncorrelated real random variables, then for $\alpha>1$ one has
    \[
    \mc{N}_\alpha(X+Y) \leq \frac{C_+(\alpha)}{C_-(\alpha)} \left( \mc{N}_\alpha(X) + \mc{N}_\alpha(Y)  \right).
    \]
\end{cor}

 \section{Reductions}

In this section we  reduce 
the inequality $h(X) \geq \frac{1}{2} \log \var (X)=1$ to variables with monotone  two-piece affine densities (see below for a precise definition). To achieve this we will make use of two standard ingredients, namely 
rearrangement together with degrees of freedom/localization techniques.

\subsection{Decreasing Rearrangement}

\begin{defn}[Decreasing Rearrangement]
    For a measurable set $A \subseteq \mathbb{R}$ 
let $|A|$ denote its Lebesgue measure and let us define
\[
    A^\downarrow = (0, |A|) \subseteq \mathbb{R}
\]
with the interpretation of $(0,0)$ as the empty set. 

\end{defn}

\begin{defn}
For a measurable function $f: \mathbb{R} \to [0,\infty]$,
 define $f^\downarrow:(0,\infty) \to [0,\infty]$
\begin{align*}
    f^\downarrow(x) = \int_0^\infty \mathbbm{1}_{\{y: f(y) > \lambda \}^\downarrow}(x) d \lambda.
\end{align*}
\end{defn}
We make some elementary observations, starting with the fact that $f^{\downarrow}$ is fully characterized by the equality $\{ f > \lambda \}^{\downarrow} = \{ f^{\downarrow} > \lambda \}$ and in particular, equimeasurable (with respect to the Lebesgue measure) functions possess identical decreasing rearrangements.  The main result of this subsection is Theorem \ref{thm: variance increases}, reduces our problem to decreasing densities.

\begin{prop}
    For $f: \mathbb{R} \to [0,\infty)$, its rearrangement $f^\downarrow$ satisfies, 
    \begin{align} \label{eq: sup rep of decreasing rearrangement}
        f^\downarrow(x) = \sup \{ \lambda : x \in \{ f > \lambda \}^\downarrow \}
    \end{align}
    and
    \begin{align} \label{eq: suplevel set characterization}
        \{ f^\downarrow > \lambda \} = \{ f > \lambda \}^\downarrow .
    \end{align}
\end{prop}

\begin{proof}
    Representation \eqref{eq: sup rep of decreasing rearrangement} follows directly from the 
    definition of $f^\downarrow$.  To prove \eqref{eq: suplevel set characterization} observe that by \eqref{eq: sup rep of decreasing rearrangement} the condition $f^\downarrow(x) > \lambda$ is equivalent to the existence
     of $\lambda' > \lambda$ such that $x \in \{ f > \lambda' \}^\downarrow$. It is therefore enough to prove the following equivalence
     \[
            \exists_{\lambda'>\lambda} \ \ x \in \{ f > \lambda' \}^\downarrow \quad \iff \quad x \in \{ f > \lambda \}^\downarrow. 
     \]  
     By the nestedness the implication $\implies$ is trivial. To show the other direction assume that  $x \in \{ f > \lambda \}^\downarrow$. This is equivalent to $x<|\{f>\lambda\}|$. By the continuity of Lebesgue measure 
     \[
      x < |\{ f > \lambda\}| = \left|  \bigcup_n \left\{ f > \lambda + \frac 1 n \right\} \right| = \lim_n |\{ f > \lambda + 1/n \}| .
     \]
     By taking $\lambda' = \lambda + \frac 1 n$ for large enough $n$ we get $\lambda' > \lambda$ such that $x \in \{ f > \lambda' \}^\downarrow$.
\end{proof}

\begin{prop}\label{prop: rearrangement pushforward same lebesgue measure}
    For $\phi$ measurable and $f$ non-negative,
    \[
        \int \phi(f(x)) dx = \int \phi(f^\downarrow(x))dx.
    \]
\end{prop}

\begin{proof}
    This is equivalent to the statement that $f$ pushes the Lebesgue measure $dm$ forward to the same measure $f^{\downarrow}$ pushes the Lebesgue measure to.  Thus it suffices to check $$f_{\#} m (\lambda, \infty) =  f_{\#}^\downarrow m (\lambda, \infty).$$  This is just $|\{ f > \lambda \}| = | \{ f^\downarrow > \lambda \}|$, which follows from the characterization $\{ f^\downarrow > \lambda \} = \{ f > \lambda \}^\downarrow$.
\end{proof}
Note that taking $\phi(x) = - x \log x$ shows that the decreasing rearrangement $f^\downarrow$ preserves the entropy of a density function $f$.

\begin{defn}
    For a random variable $X$ with density function $f$, define $X^{\downarrow}$ to be a random variable drawn from the density function $f^\downarrow$.
\end{defn}


When $X$ has density $f$ we will write the variance as $\Var(X)$ and $\Var(f)$ interchangeably.

\begin{lem}\label{lem: decreasing rearrangement decreases moments}
    For non-negative $X$, and increasing, non-negative $\phi$,
    \[
        \mathbb{E}\phi(X^{\downarrow}) \leq \mathbb{E}\phi(X).
    \]
\end{lem}

\begin{proof}
    The proof follows from the elementary observation that for $a \geq 0$, 
    \begin{align} \label{eq: decreasing rearrangement has smaller tails}
        \left| (a, \infty) \cap A \right| \geq \left| (a,\infty) \cap A^{\downarrow} \right|.
    \end{align}
    Using the layer-cake decomposition,
    \begin{align*}
        \mathbb{E}\phi(X) 
            &=
                \int_0^\infty  \int_0^\infty \left( \int_0^\infty \mathbbm{1}_{\{ \phi > \lambda \}}(x) \mathbbm{1}_{\{ f > \sigma \}}(x) dx \right) d\lambda d\sigma
                    \\
            &=
                \int_0^\infty  \int_0^\infty  \left| \{ \phi > \lambda \}  \cap \{ f > \sigma \} \right|d\lambda d\sigma
                    \\
            &\geq
                \int_0^\infty  \int_0^\infty  \left| \{ \phi > \lambda \}  \cap \{ f^\downarrow > \sigma \} \right|d\lambda d\sigma
                    =
                \mathbb{E}\phi(X^{\downarrow}),
    \end{align*}
    where the inequality follows from \eqref{eq: decreasing rearrangement has smaller tails}.
\end{proof}

\begin{lem} \label{lem: variance decomposition}
    For probability densities $f_0, f_1$ and $\lambda \in [0,1]$ the density $f = \lambda f_0 + (1-\lambda) f_1$ has variance,
    \[
        \Var(f) =  \lambda \Var(f_0) + (1-\lambda) \Var(f_1) + \lambda (1-\lambda) (\mu_1 - \mu_0)^2
    \]
    where $\mu_i$ denotes the barycenter of $f_i$, that is $\mu_i = \int x f_i(x) dx$.
\end{lem}

\begin{proof}
It holds
    \begin{align*}
       & \Var(f) = \int x^2 ( \lambda  f_0(x) + (1-\lambda) f_1(x) ) dx - \left( \int x (\lambda  f_0(x) + (1-\lambda) f_1(x) ) dx \right)^2
        \\
            =&
                \lambda \left(  \int x^2 f_0(x) dx -   \left( \int x f_0(x) dx \right)^2 \right) + (1- \lambda)  \left(  \int x^2 f_1(x) dx -   \left( \int x f_1(x) dx \right)^2 \right)
                    \\
            &+
                \lambda \left( \int x f_0(x) dx \right)^2 + (1-\lambda) \left( \int x f_1(x) dx \right)^2 - \left( \int x (\lambda f_0(x) + (1-\lambda) f_1(x) ) dx \right)^2
                    \\
            =&
                \lambda \Var(f_0) + (1-\lambda) \Var(f_1) +  \lambda (1-\lambda) (\mu_1 - \mu_0)^2.
    \end{align*}
\end{proof}

\begin{thm} \label{thm: variance increases}
    For $X$ log-concave,
    \[
        \Var(X) \leq \Var(X^{\downarrow}).
    \]
\end{thm}

\begin{proof}
    We prove the result when $X$ has a density $f$ given by a unimodal step function by induction, that is $f = \sum_{k=0}^n \lambda_k \mathbbm{1}_{I_k}/|I_k|$ with $I_k$ intervals satisfying $I_{k+1} \subsetneq I_k$  and $\lambda_k > 0$.  An easy limiting argument gives the result for log-concave $X$.  When $n = 0$, $X$ is uniform and the result is immediate.  Assuming the result for $n' < n$, we proceed. 
 The density of $f$ can be written as $\lambda f_0 + (1-\lambda) f_1$, with $\lambda = \lambda_0$, $f_0 = \frac{\mathbbm{1}_{I_0}}{|I_0|}$, and $f_1 = \sum_{k=1}^n \frac{\lambda_k}{1-\lambda_0} \frac{\mathbbm{1}_{I_k}}{|I_k|}$.  Observing that $f_1$ takes strictly less  values than $f$ and that by affine invariance of the inequality we may assume that $I_0 = (0,1)$, and by considering $\tilde{X} = 1 - X$, we may assume without loss of generality, that $\int x f_1(x)dx \leq \frac 1 2$.  \\

 By Lemma \ref{lem: variance decomposition},
 \[
    \Var(f) = \lambda \Var(f_0) + (1-\lambda) \Var(f_1) + \lambda (1-\lambda) \left(\frac 1 2 - \int x f_1(x) dx \right)^2.
 \]
 Observe that $f^{\downarrow} = \lambda f_0^\downarrow + (1-\lambda) f_1^{\downarrow} $. Applying Lemma \ref{lem: variance decomposition} to $f^\downarrow$,
 \[
 \Var(f^\downarrow) = \lambda \Var(f_0^\downarrow) + (1-\lambda) \Var(f_1^\downarrow) + \lambda (1-\lambda) \left(\frac 1 2 - \int x f_1^\downarrow(x) dx \right)^2.
 \]
Clearly $f_0^\downarrow = f_0$. By induction $\Var(f_1) \leq \Var(f_1^{\downarrow})$ and by Lemma \ref{lem: decreasing rearrangement decreases moments}, applied to $\phi(x) = x$, we have $\frac 1 2 \geq \int x f_1(x) dx \geq \int x f_1^\downarrow(x) dx$. The result follows.
\end{proof}

\begin{lem} \label{lem: decreasing rearrangement decreases ratio}
    For $X$ log-concave,
    \[
        \frac{e^{2 h(X)}}{\Var(X)} \geq \frac{e^{2h(X^{\downarrow})}}{\Var(X^\downarrow)}
    \]
\end{lem}
\noindent This follows immediately from Theorem \ref{thm: variance increases} and Proposition \ref{prop: rearrangement pushforward same lebesgue measure} (with $\phi(x) = - x \log x$) since decreasing rearrangement will preserve entropy and increase variance. It allows us to reduce our problem to $X$ log-concave with monotone decreasing density, in light of the following.

\begin{prop}
    For $X$ log-concave, $X^{\downarrow}$ is log-concave as well.
\end{prop}
This is an immediate corollary of a more general result, Proposition 7.5.8 in \cite{melbourne2019rearrangement}.  We include the proof here for the convenience of the reader. 
\begin{proof}
    Note that $X$ has a log-concave density $f$ if and only if 
    \begin{align} \label{eq: set theoretic inclusion}
          (1-t) \{ f > \lambda_1 \} + t \{ f > \lambda_2 \} \subseteq
          \{ f > \lambda_1^{1-t} \lambda_2^t \}.
    \end{align}
    Thus, to show that $X^\downarrow$ is log-concave it suffices to prove
    \[
         (1-t) \{ f > \lambda_1 \}^\downarrow + t \{ f > \lambda_2 \}^\downarrow \subseteq 
        \{ f > \lambda_1^{1-t} \lambda_2^t \}^\downarrow
    \]
    But since both sets are open intervals it suffices to prove that the right hand side has bigger volume, which follows Brunn-Minkowski inequality and the set theoretic inclusion in \eqref{eq: set theoretic inclusion},
    \begin{align*}
         | (1-t) \{ f > \lambda_1 \}^\downarrow + t \{ f > \lambda_2 \}^\downarrow | 
            & \leq 
                | (1-t) \{ f > \lambda_1 \} + t \{ f > \lambda_2 \}| 
                    \\
            &\leq |\{ f > \lambda_1^{1-t} \lambda_2^t \}| 
                    \\
                &= 
                |\{ f > \lambda_1^{1-t} \lambda_2^t \}^\downarrow|.
    \end{align*}
\end{proof}

\subsection{Degrees of freedom} Our goal is to show that in order to prove Theorem \ref{thm:main} it suffices to consider functions of the form $e^{-V}$ where $V$ is a two-piece affine on a finite interval. This is a standard argument appearing e.g. in \cite{MNT21}, so we only sketch it. 

\emph{Step 1.} Let $\sigma^2 = \var(X)$. By the previous subsection, in order to prove the inequality $h(f) \geq \log(e \sigma)$ it suffices to consider non-increasing densities on $[0,\infty)$. In fact by an approximation argument we can assume that the support of $f$ is a finite interval. Indeed if we define $f_n = e^{-V}\mathbbm{1}_{(0,n)}/c_n$, where $c_n = \int_0^n e^{-V}$ is the normalizing constant, then clearly $c_n \to 1$ as $n \to \infty$ and by the Lebesgue dominated convergence theorem $\var(f_n)\to \var(f)$. We also have
\[
    h(f_n) = \frac{1}{c_n} \int_0^n V e^{-V} + \frac{\log c_n}{c_n} \int_0^n  e^{-V} \to h(f)
\]
by the Lebesgue dominated convergence theorem, as $|V|e^{-V} =|V|e^{-V/2} e^{-V/2} \leq e^{-V/2}$ is integrable.  

\emph{Step 2.} We can therefore fix the interval $[0,L]$ on which the function $f$ is defined. Let $A=A_{L,\sigma}$ denote the set of log-concave non-increasing densities supported in $[0,L]$ and having variance $\sigma^2$. We have $f(0)^2 \var(X) \leq f(0)^2 \mE X^2 \leq 2$, see \cite{Ball88}. To see this we follow the argument from \cite{MNT21}. By scaling we can assume that $f(0)=1$. If $g(x)=e^{-x}\mathbbm{1}_{[0,\infty)}(x)$ then by log-concavity of $f$ the function $f-g$ changes sign in exactly one point $a>0$ and thus
\[
    \mE X^2-2 = \int_0^\infty x^2 (f(x)-g(x)) \dd x = \int_0^\infty (x^2-a^2) (f(x)-g(x)) \dd x  \leq 0
\]
since the integrant is non-positive.
This shows that $f$ is bounded by $2 \sigma^{-2}$ and in particular $h(f) =-\int f\log f \geq -\log f(0)$ is bounded from below. This shows that the quantity $M=\inf \{h(f): f \in A_{L,\sigma}\}$ is finite. Let $f_n$ be such that $h(f_n) \to M$. By a straightforward adaptation of Lemma 12 from \cite{MNT21} we get that $(f_n)$ has a convergent subsequence $f_{n_k} \to f_0 \in A$ and by the Lebesgue dominated convergence theorem $h(f_0)=M$. This shows that the infimum of $h(f)$ is attained on $A$.    

\emph{Step 3.} We now apply the theory of degrees of freedom due to Fradelizi and Gu\'edon from \cite{fradelizi2006generalized}. We say that $f \in A$ has $d$ degrees of freedom if there exist $\e>0$ and linearly independent functions $f_1,\ldots, f_d$ such that $f_\delta := f+\sum_{i=1}^d \delta_i f_i$ is log-concave and non-incresing for all $|\delta_i| \leq \e$. Suppose $f$ has at least $4$ degrees of freedom. Then 
\[
\int f_\delta(x) \dd x= \int f(x) \dd x, \qquad \int x f_\delta(x) \dd x = \int x f(x) \dd x, \qquad \int x^2 f_\delta(x) \dd x = \int x^2 f(x) \dd x
\]
is a system of $3$ linear equations in variables $\delta_1, \ldots, \delta_4$ and thus has nontrivial linear subspace of solutions. Thus there is $\delta$ such that $f_\delta, f_{-\delta} \in A$ and 
\[
    h(f) = h( \frac12 f_\delta + \frac12 f_{-\delta} ) >  \frac12 h(f_\delta) + \frac12 h(f_{-\delta})
\]
since entropy is a strictly convex functional. Thus $f$ is not the extremizer of $h(f)$. This shows that extremizers have to have at most $3$ degrees of freedom and  Step IV of the proof of Theorem 1 from \cite{MNT21} shows that such functions must be piecewise log-affine with at most two pieces. 

\subsection{Localization}

Alternatively, we can proceed with extreme point analysis following Fradelizi and Guedon \cite{fradelizi2006generalized}.
For a fixed compact interval $K \subseteq \mathbb{R}$,  and upper semi-continuous functions $g_1, g_2$ define $\mathcal{P}_g$ to be the space of log-concave probability measures $\mu$ supported in $K$ such that 
\[
    \int g_i d\mu \geq 0.
\]
We will use the following special case of Fradelizi and Guedon.
\begin{thm}[\cite{fradelizi2006generalized}{ Theorem 1}] \label{thm: fradelizi-guedon}
    Let $\nu$ be an extreme point of the convex hull of $\mathcal{P}_g$, then $\nu$ is a point mass, or $\nu$ has density $e^{-V}$ with respect to the Lebesgue measure where on the support of $\nu$,  $V = \max \{ \phi_1, \phi_2 \}$ for $\phi_i$ affine.
\end{thm}
When a density $f$ has the form $e^{-V}$ for $V = \max \{ \phi_1, \phi_2 \}$ for $\phi_i$ affine, we will say that $f$ is two piece log-affine.

\begin{lem}
    For $X$ log-concave on $\mathbb{R}$,
    \[
        \frac{e^{2 h(X)}}{\Var(X)} \geq \inf_{Z \in \mathcal{K}} \frac{e^{2h(Z)}}{\Var(Z)},
    \]
    where $\mathcal{K}$ is the space of compactly supported log-concave variables with monotone density on this support of the form $f = e^{ - \max \{ \phi_0, \phi_1 \}}$ for $\phi_i$ affine.
\end{lem}

\begin{proof}
    Recalling the truncation argument  it suffices to consider $X \sim \mu$ with density supported on $[0,L]$ for some $L > 0$.  Fix $X$ and take
    \begin{align*}
        g_1(x) &= \mathbb{E}[X] - x
        \\
     g_2(x) &= x^2 - \mathbb{E}[X^2].
     \end{align*}
    For $Z \sim \nu$ an extreme point of $\mathcal{P}_g$, since $Z$ is non-negative by $\int g_i d\nu \geq 0$, we have
    \[
    \Var(Z) \geq \Var(X) > 0,
    \]
    so $\nu$ is not a point mass and hence by Theorem \ref{thm: fradelizi-guedon}, $Z$ has density of the form $f = e^{- \max\{ \phi_1 , \phi_2 \}}$.  Since $X \sim \mu \in \mathcal{P}_g$  by definition, if we let $\mathcal{E}(\mathcal{P}_f)$ denote the extreme points of the convex hull of $\mathcal{P}_f$, by the Krein-Milman $\mu$ belongs to the closure of the convex hull of $\mathcal{E}(\mathcal{P}_f)$.  Now let us show that this implies that 
    \begin{align} \label{eq: entropy minimizer}
        h(X)  \geq \inf_{ Z \sim \nu \in \mathcal{E}(\mathcal{P}_g)} h(Z).
    \end{align}
    Indeed the entropy is concave and upper semi-continuous in the weak topology when restricted to compact sets  (as can be seen from the more well known fact that the relative entropy is lower semicontinuous, and on compact sets $h(X) = h(U) - D(X||U)$ where $U$ is the uniform distribution on the compact set), thus writing $\mu$ as $\lim_n \mu_n$ for a sequence of $\mu_n$ that can be expressed as convex combination of extreme points, that is $ \mu_n = \sum_{i =1}^{k_n} \lambda_n(i) \nu_{n}(i)$ for $\nu_{n}(i) \in \mathcal{E}(\mathcal{P}_g)$, we have
    \begin{align*}
        h(\mu) 
        & \geq 
        \limsup_{n} h(\mu_n) \\
        & = 
        \limsup_{n} h\left( \sum_{i =1}^{k_n} \lambda_n(i) \nu_{n}(i) \right) \\
        & \geq 
        \limsup_n \sum_{i =1}^{k_n} \lambda_n(i) h(\nu_{n}(i)) \\
        & \geq \inf_{\nu \in \mathcal{E}(\mathcal{P}_g)} h(\nu).
    \end{align*}Since every element of $\mathcal{P}_g$ has variance no smaller than $X$, it follows that 
    \[
    \frac{e^{2h(X)}}{\Var(X)} \geq \inf_{Z \sim \nu \in \mathcal{E}(\mathcal{P}_g)} \frac{e^{2h(Z)}}{\Var(Z)}.
    \]
    Consider $Z^\downarrow$ for $Z \in \mathcal{E}(\mathcal{P}_g)$, and applying Lemma \ref{lem: decreasing rearrangement decreases ratio}, we have
    \[
        \frac{e^{2h(X)}}{\Var(X)} \geq \inf_{Z \sim \nu \in \mathcal{E}(\mathcal{P}_g)} \frac{e^{2h(Z^{\downarrow})}}{\Var(Z^\downarrow)}
    \]
    Direct computation, shows that if $Z$ has a two-piece log-affine density, then $Z^\downarrow$ does as well, completing the proof.
\end{proof}

\section{Proof of Theorem \ref{thm:main}}

The proof of Theorem \ref{thm:main} goes in 3 steps. The above preliminaries first allow us to reduce to a two-piece affine function which, after re-parametrisation, amounts to proving that a 3 variables function $G(c,x,y)$ is non-negative.  As a second technical step we remove an exponential factor to reduce to a simpler function that we finally study by hand.

\subsection{Three-point inequality} From the previous section we can assume that $f=e^{-V}$, where $V$ is two-piece affine and non-decreasing on an interval $[0,L]$. In fact by scale invariance   we can assume that we have the following parametrization of our function
\[
	g(t) =  e^{- \frac{t}{a}}\mathbbm{1}_{[-ax,0]}(t) + e^{-\frac{t}{b}} \mathbbm{1}_{[0,-yb]}(t), \qquad f=\frac{g}{c}, \qquad c = \int g = a(e^x-1)-b(e^y-1),
\]
where
\[
	a \geq b >0, \qquad x\geq 0, \qquad y \leq 0 .
\]
Then $f$ is a probability density and we have
\[
	\int x g(x) \dd x = a^2 \left(e^x (1-x)-1\right)-b^2 \left(e^y (1-y)-1\right)
\]
and
\[
	\int x^2 g(x) \dd x = a^3 \left(e^x \left(x^2-2 x+2\right)-2\right)-b^3 \left(e^y \left(y^2-2 y+2\right)-2\right).
\]
Also
\[
	- \int g \log g = a \left(e^x (1-x)-1\right)-b( e^y (1-y)-1).
\]
The inequality of Theorem \ref{thm:main} we want to prove is
\[
	-\int f \log f \geq \frac12 \log \var(f) + 1 .
\]
This is
\[
	e^{-2 \int f \log f} \geq e^2 \var(f).
\]
In terms of $g$
\[
	e^{-2} e^{-2 \int \frac{g}{c} \log \frac{g}{c}} \geq  \frac{1}{c}\int x^2 g(x) \dd x -  \frac{1}{c^2} \left(\int x g(x) \dd x \right)^2.
\]
We have
\[
e^{-2} e^{-2 \int \frac{g}{c} \log \frac{g}{c}} = e^{-2} e^{-\frac{2}{c} \int g \log g + 2 \log c}  = e^{-2} c^2 e^{-\frac{2}{c} \int g \log g},
\]
so after multiplying both sides by $c^2$ we want to prove
\[
	e^{-2} c^4 e^{-\frac{2}{c} \int g \log g} \geq  c\int x^2 g(x) \dd x -  \left(\int x g(x) \dd x \right)^2.
\]
Observe that
\begin{align*}
    e^{-2}  e^{-\frac{2}{c} \int g \log g} 
    & = 
    \exp\left( 2 \frac{ a \left(e^x (1-x)-1\right)-b( e^y (1-y)-1)}{a(e^x-1)-b(e^y-1)} -2\right) \\
    & = 
    \exp\left( 2 \frac{ -a x e^x +b y e^y}{a(e^x-1)-b(e^y-1)} \right).
\end{align*}
Our goal is therefore to prove
\[
	c^4 \exp\left( 2 \frac{ -a x e^x +b y e^y}{a(e^x-1)-b(e^y-1)} \right) \geq  c \int x^2 g(x) \dd x -  \left(\int x g(x) \dd x \right)^2.
\]
Equivalently
\begin{align*}
	& \left( a(e^x-1)-b(e^y-1) \right)^4 \exp\left( -2 \frac{ a e^x x -b  e^y y}{a(e^x-1)-b(e^y-1)} \right) \\
	&  \qquad \geq  \left( a(e^x-1)-b(e^y-1) \right) \left( a^3 \left(e^x \left(x^2-2 x+2\right)-2\right)-b^3 \left(e^y \left(y^2-2 y+2\right)-2\right) \right) \\
	& \qquad \qquad  -  \left(  a^2 \left(e^x (1-x)-1\right)-b^2 \left(e^y (1-y)-1\right) \right)^2.
\end{align*}

The inequality is invariant under $(a,b) \to (ta,tb)$ and therefore we can assume that $b=1$ and $a \geq 1$.
Let us define the function
\begin{align*}
	G(c,x,y) = & \left( c(e^x-1)-(e^y-1) \right)^4 \exp\left( -2 \frac{ c e^x x -  e^y y}{c(e^x-1)-(e^y-1)} \right) \\
	&  \qquad -  \left( c(e^x-1)-(e^y-1) \right) \left( c^3 \left(e^x \left(x^2-2 x+2\right)-2\right)- \left(e^y \left(y^2-2 y+2\right)-2\right) \right) \\
	& \qquad \qquad  +  \left(  c^2 \left(e^x (1-x)-1\right)- \left(e^y (1-y)-1\right) \right)^2.
\end{align*}
All together, we reduced the proof of Theorem \ref{thm:main} to  proving that $G(c,x,y) \geq 0$ for $x \geq 0$, $y \leq 0$ and $c \geq 1$.

\subsection{Estimating the exponent}

We observe that
\[
	\frac{ce^xx-e^yy}{c(e^x-1)-(e^y-1)}- x = \frac{(x-y)e^y+x(c-1)}{c(e^x-1)-(e^y-1)} \in [0,1].
\]
Indeed, the extremal values of the right hand side are attained for $c=0,1$ and are therefore equal $\frac{(x-y)e^y}{e^x-e^y}=\frac{(x-y)}{e^{x-y}-1}$ and $\frac{x}{e^x-1}$. 
Since clearly $\frac{\theta}{e^\theta-1} \in [0,1]$ for any $\theta >0$ the claim follows.

Now, on the interval $[0,2]$ we have the bound $e^{-x} \geq 1-x+\frac12 x^2 -\frac16 x^3 +\frac{7}{240} x^4$. The inequality holds for $x=2$ and thus it is enough to show that the derivative of $g(x)=e^{-x} -(1-x+\frac12 x^2 -\frac16 x^3 +\frac{7}{240} x^4)$ is positive or has sign pattern $(+,-)$. It is clear that if $f:\R_+ \to \R$ satisfies $f(0)=0$ and $f'$ is positive or has sign pattern $(+,-)$ then $f$ itself is positive or has sign pattern $(+,-)$. Using this argument three times and observing that $g^{(4)}(x)=e^{-x}-\frac{7}{10}$ has sign pattern $(+,-)$ finishes the proof. Using this bound we obtain
\[
	e^{-2\frac{ce^xx-e^yy}{c(e^x-1)-(e^y-1)}} \geq e^{-2x}\left( 1-L+\frac12 L^2 -\frac16 L^3 +\frac{7}{240} L^4 \right)
 \]
 with 
 $L(c,x,y)= 2\frac{(x-y)e^y+x(c-1)}{c(e^x-1)-(e^y-1)}$. 
 This gives a lower bound on $G(c,x,y)$ of the form
\[
	15 e^x G(c,x,y) \geq P_0(x,y)+(c-1)P_1(x,y) + (c-1)^2P_2(x,y) + (c-1)^3P_3(x,y) + (c-1)^4P_4(x) 
\]
where $P_0,\dots, P_4$ are explicitely given and studied in the next sections.

To prove Theorem \ref{thm:main} it is therefore
sufficient to prove that $P_i(x,y) \geq 0$ for $0 \leq i \leq 4$ which we now show.

\subsection{Positivity of $P_i(x,y)$, $i=0,1,\dots, 4$}

To prove that the $P_i$'s are positive is a tedious but somehow simple exercise. The strategy that we adopt rely on a series expansion in the variable $x$. Some details are left to the reader for shortness.

We start with the simplest case $P_4$ that involves only the variable $x$.

\subsubsection{Positivity of $P_4(x,y)$}
We have
\begin{align*}
	P_4(x) & = 7 x^4+20 x^3+30 x^2 + 30 x+15 +15 e^{3 x} \left(x^2-2 x-2\right) +15 e^{2 x} \left(2 x^2+6 x+5\right) \\ 
	& \quad   -10 e^x \left(2 x^3+6 x^2+9 x+6\right).
\end{align*}
Expanding in $x$ we write $P_4(x)= \sum_{n \geq 0} \frac{a_n}{n!} x^n$, where
\[
	a_n = -10 \left(2 n^3+7 n+6\right)+15\cdot 2^{n-1} \left(n^2+5 n+10\right)+5\cdot 3^{n-1} (n-9) (n+2)   , \qquad n \geq 5
\] 
and $a_0=a_1=a_2=a_3=0$ and $a_4=\frac34$. It is easy to check that $a_n \geq 0$ for all $n \geq 0$, as for large $n$ the term $5 \cdot 3^{n-1} n^2$ dominates and for small $n$ the inequality can be checked directly.  

\subsubsection{Positivity of $P_0(x,y)$}
We have 
\begin{align*}
	& e^{-y}P_0(x,y)  = \\
 & \qquad 15 e^{2 x+y} \left(2 x^2+x (6-4 y)+2 y^2-6 y+5\right)+15 e^{3 x} \left(x^2-2 x (y+1)+y^2+2 y-2\right) \\
	& \qquad -10 e^{x+2 y} \left(2 x^3-6 x^2 (y-1)+3 x \left(2 y^2-4 y+3\right)-2 y^3+6 y^2-9 y+6\right) \\
	& \qquad +e^{3 y} \left(7 x^4-4 x^3 (7 y-5)+6 x^2 \left(7 y^2-10 y+5\right)+x \left(-28 y^3+60 y^2-60 y+30\right) \right) \\
	& \qquad + e^{3y}\left(7 y^4-20 y^3+30 y^2-30 y+15\right).
\end{align*}
We Taylor expand the right hand side in the form $\sum_{n \geq 0} \frac{f_n(y)}{n!} x^n$, where
\[
	f_n(y) = a_0 + a_1y+a_2 y^2 + e^y(b_0+b_1y+b_2y^2) + e^{2y}(c_0+c_1y+c_2y^2+c_3y^3), \qquad n \geq 5,
\]
with
\[
	a_0(n) = 3^{n-1} \left(5 n^2-35 n-90\right), \quad a_1(n)=3^n \cdot 10 \cdot (3- n), \quad a_2(n)=5 \cdot 3^{n+1}
\]
\[
	b_0(n)= 15 \cdot 2^{n-1} \left(n^2+5 n+10\right), \quad b_1(n)= -15 \cdot 2^{n+1} (n+3), \quad b_2(n)=15 \cdot  2^{n+1},
\]
\[
	c_0(n)=-10 \left(2 n^3+7 n+6\right), \quad c_1(n)=30 \left(2 n^2+2 n+3\right), \quad c_2(n)=-60 (n+1), \quad c_3(n)=20. 
\]
The terms with $a_1, a_2, b_0, b_1, b_2$ are clearly positive. Using crude bounds $e^{2k}|y|^l \leq 1$ for $k=1,2$ and $l=0,1,2,3$  we arrive at $f_n(y) \geq a_0+c_0-c_1+c_2-c_3$. The leading term in this expression is $5 \cdot 3^{n-1} n^2$ and it is easy to check that the expression is positive for $n \geq 10$.

The inequalities $f_n(y) \geq 0$ for $0 \leq n \leq 9$ can be proved by Taylor expanding $e^{2t}f_n(-t)$ and checking that the coefficients are nonnegative. We leave the details to the reader.

\subsubsection{Positivity of $P_1(x,y)$}

We have 
\begin{align*}
& P_1(x,y)  =15 e^{2 x+y} \left(4 x^2-4 x (y-3)-y^2-8 y+8\right)+60 e^{2 (x+y)} \left(x^2+x (3-2 y)+y^2-3 y+3\right) \\
& \quad +15 e^{3 x+y} \left(3 x^2-4 x (y+2)+y^2+8 y-8\right)+15 e^{3 x} x^2 \\
& \quad -30 e^{x+2 y} \left(2 x^3+x^2 (6-4 y)+x \left(2 y^2-8 y+9\right)+2 \left(y^2-3 y+3\right)\right) \\
& \quad  -10 e^{x+3 y} \left(2 x^3-6 x^2 (y-1)+3 x \left(2 y^2-4 y+3\right)-2 y^3+6 y^2-9 y+6\right) \\
& \quad  -2 e^{3 y} \left(-14 x^4+x^3 (42 y-40)-6 x^2 \left(7 y^2-15 y+10\right)+2 x \left(7 y^3-30 y^2+45 y-30\right) \right) \\
& \quad -10 e^{3 y} \left(2 y^3-6 y^2+9 y-6\right) .
\end{align*}
We Taylor expand the right hand side in the form $\sum_{n \geq 0} \frac{f_n(y)}{n!} x^n$, where
\[
	f_n(y) = a_0 +  e^y(b_0+b_1y+b_2y^2) + e^{2y}(c_0+c_1y+c_2y^2) + e^{3y}(d_0+d_1 y+d_2 y^2 + d_3 y^3), \qquad n \geq 5,
\]
with
\[
	a_0(n)=5 \cdot 3^{n-1} (n-1) n, \quad  b_0(n)= 5 \cdot 3^n \left( n^2-9 n-24\right)+5 \cdot 2^n \left(3 n^2+15 n+24\right), \quad 
\]
\[
	b_1(n)=-20 \cdot 3^n (n-6)-30 \cdot 2^n (n+4), \quad b_2(n)=15 \left(3^n-2^n\right)
\]
\[
	c_0(n)=15\cdot 2^n \left(n^2+5 n+12\right)-30 \left(2 n^3+7 n+6\right), \quad c_1(n)=60 \left(2 n^2+2 n+3\right)-15\ 2^{n+2} (n+3)
\]
\[
	c_2(n)=15\cdot 2^{n+2}-60(n+1), \quad d_0(n)=-10 \left(2 n^3+7 n+6\right), \quad d_1(n)=30 \left(2 n^2+2 n+3\right), 
\]
\[
d_2(n)=-60 (n+1), \quad d_3(n)=20.
\]
We always have $a_0,b_2,c_0,c_2,d_1,d_3>0$ and $b_1, c_1, d_0, d_2<0$, whereas $b_0>0$ only for $n \geq 11$. For $n \geq 11$ we therefore have $f_n(y) \geq a_0 +d_0-d_1+d_2-d_3>0$, the dominating term being $a_0$. For $9 \leq n \leq 11$ we get $f_n(y) \geq a_0 +b_0+d_0-d_1+d_2-d_3>0$.  

The inequalities $f_n(y) \geq 0$ for $0 \leq n \leq 8$ can be proved by Taylor expanding $e^{3t}f_n(-t)$ and checking that the coefficients are nonnegative. 

\subsubsection{Positivity of $P_2(x,y)$}
We have
\begin{align*}
	& \frac13 P_2(x,y)  = 10 e^{2 (x+y)} \left(x^2+x (3-2 y)+y^2-3 y+3\right)+10 e^{2 x+y} \left(4 x^2-4 x (y-3)-7 y+10\right) \\
	& \quad +10 e^{2 x} \left(x^2+3 x+2\right)+5 e^{3 x} \left(3 x^2-2 x-4\right) \\
	& \quad  -10 e^{x+2 y} \left(2 x^3+x^2 (6-4 y)+x \left(2 y^2-8 y+9\right)+2 \left(y^2-3 y+3\right)\right) \\
	& \quad  -10 e^{x+y} \left(2 x^3-2 x^2 (y-3)+x (9-4 y)-3 y+6\right) \\
	& \quad +2 e^{2 y} \left(7 x^4+x^3 (20-14 y)+x^2 \left(7 y^2-30 y+30\right)+10 x \left(y^2-3 y+3\right)+5 \left(y^2-3 y+3\right)\right) \\
	& \quad  +5 (x-4) e^{3 x+y} (3 x-2 y+2).
\end{align*}
We Taylor expand the right hand side in the form $\sum_{n \geq 0} \frac{f_n(y)}{n!} x^n$, where
\[
	f_n(y) = a_0 + e^y(b_0+b_1y) + e^{2y}(c_0+c_1y+c_2y^2), \qquad n \geq 5,
\]
with
\[
	a_0(n)=15\cdot 2^{n-1} \left(n^2+5 n+8\right)+5\cdot 3^n \left(n^2-3 n-12\right), 
\]
\[
 b_0(n)=-30 \left(2 n^3+7 n+6\right)+15\cdot 2^{n+1} \left(n^2+5 n+10\right)+5\cdot 3^n \left(n^2-11 n-24\right)
\]
\[
	b_1(n) = -10\cdot 3^n (n-12)-15\cdot 2^{n+1} (2 n+7)+30(2n^2+2n+3)
\]
\[
	c_0(n)=15\cdot 2^{n-1} \left(n^2+5 n+12\right)-30 \left(2 n^3+7 n+6\right), 
\] 
\[
c_1(n)=60 \left(2 n^2+2 n+3\right)-15\cdot 2^{n+1} (n+3), \qquad c_2(n)=15\cdot 2^{n+1}-60(n+1).
\]
We always have $a_0,c_0,c_2>0$ and $c_1<0$, whereas $b_0>0$ only if $n \geq 13$ and $b_1<0$ only if $n \geq 11$. For $n \geq 13$ we have $f_n(y) \geq a_0-b_1>0$. Here the dominating terms is $5 \cdot 3^n n^2$ from $a_0$. For $7 \leq n \leq 12$ we have $f_n(y) \geq a_0-|b_0|-|b_1|>0$. 

When $0 \leq n \leq 6$ it is enough to Taylor expand $e^{2t}f_n(-t)$ and check that the coefficients are nonnegative. 

\subsubsection{Positivity of $P_3(x,y)$}

We have
\begin{align*}
	P_3(x,y)  & = 15 e^{3 x+y} \left(x^2-4 x+2 y-2\right)+30 e^{2 x+y} \left(2 x^2-2 x (y-3)-3 y+5\right) \\
 & \quad  +30 e^{2 x} \left(2 x^2+6 x+5\right)
 +15 e^{3 x} \left(3 x^2-4 x-6\right)\\
	& \quad -30 e^{x+y} \left(2 x^3-2 x^2 (y-3)+x (9-4 y)-3 y+6\right)   -10 e^x \left(2 x^3+6 x^2+9 x+6\right) \\
	& \quad +e^y \left(28 x^4+x^3 (80-28 y)-60 x^2 (y-2)-60 x (y-2)-30 (y-2)\right),
\end{align*}
where
\[
	f_n(y) = a_0 + e^y(b_0+b_1y), \qquad n \geq 5.
\]
with
\[
	a_0(n) = -10 \left(2 n^3+7 n+6\right)+15\cdot 2^n \left(n^2+5 n+10\right)+5\cdot 3^n \left(n^2-5 n-18\right),
\]
\[
	b_0(n)= -30 \left(2 n^3+7 n+6\right)+15\cdot 2^n \left(n^2+5 n+10\right)+5\cdot 3^{n-1} \left(n^2-13 n-18\right)
\]
\[
	b_1(n)= 30 \left(2 n^2+2 n+3\right)-15\cdot 2^{n+1} (n+3)+10\cdot 3^{n+1}.
\]
We always have $a_0, b_1>0$ whereas $b_0>0$ only for $n \geq 14$. For $n \geq 14$ we get $f_n(y) \geq a_0-b_1>0$, the dominating term being $5 \cdot 3^n n^2$ from $a_0$. For $6 \leq n \leq 13$ we write $f_n(y) \geq a_0-|b_0|-|b_1| >0$. 

When $0 \leq n \leq 5$ it is enough to Taylor expand $e^{t}f_n(-t)$ and check that the coefficients are nonnegative. 

\section*{Acknowledgements}
The authors express their gratitude to Krzysztof Oleszkiewicz for lengthy discussions, suggestions, and support for this project, without which this paper would not exist, the Institut Henri Poincaré and Université Paris Nanterre, where the majority of this work was completed, for their warm hospitality, and for generously sharing their thoughts, motivations, and suggestions about this problem, Sergey Bobkov, Mokshay Madiman, and Arnaud Marsiglietti.

\bibliographystyle{plain}
\bibliography{bibibi}

\end{document}